\documentclass{amsart}
\usepackage[margin=1.5in]{geometry}
\usepackage{amssymb,amsmath,array,multirow,makecell,blindtext,amsthm,enumitem,mathtools,amscd,tikz-cd,amsrefs,dsfont,hyperref,url,caption,float,placeins,bm,color,mathrsfs,comment,float}
\usepackage[font=small, labelfont=bf]{caption}
\hypersetup{
    colorlinks=true,
    linkcolor=blue,
    citecolor=black,
    filecolor=magenta,      
    urlcolor=cyan,
}
\usepackage[all]{xy}
\usepackage{graphicx}
\usetikzlibrary{positioning}
\usepackage{caption}
\usepackage{graphicx}

 \DeclareFontFamily{U}{wncy}{}
    \DeclareFontShape{U}{wncy}{m}{n}{<->wncyr10}{}
    \DeclareSymbolFont{mcy}{U}{wncy}{m}{n}
    \DeclareMathSymbol{\Sha}{\mathord}{mcy}{"58}

    \definecolor{ForestGreen}{RGB}{34,139,34}
\newcommand{\green}[1]{\textcolor{ForestGreen}{#1}}

\theoremstyle{plain}

\newtheorem{theorem}{Theorem}[section]
\newtheorem{corollary}[theorem]{Corollary}
\newtheorem{lemma}[theorem]{Lemma}
\newtheorem{remark}[theorem]{Remark}
\newtheorem{proposition}[theorem]{Proposition}
\newtheorem{definition}[theorem]{Definition}
\newtheorem{defn}[theorem]{Definition}

\newtheorem{example}[theorem]{Example}
\numberwithin{theorem}{section}

\newcommand{\Z}{\mathbb{Z}}
\newcommand{\Q}{\mathbb{Q}}
\newcommand{\QQ}{\mathbb{Q}}
\newcommand{\Qp}{\mathbb{Q}_p}

\newcommand{\R}{\mathbb{R}}
\newcommand{\C}{\mathbb{C}}

\newcommand{\Oo}{\mathcal{O}}
\newcommand{\Gg}{\mathcal{G}}

\newcommand{\Gal}{\operatorname{Gal}}

\newcommand{\sgn}{\operatorname{sgn}}

\newcommand{\Hom}{\operatorname{Hom}}
\newcommand{\Div}{\operatorname{Div}}

\newcommand{\ord}{\operatorname{ord}}

\newcommand{\GL}{\operatorname{GL}}

\newcommand{\nf}{\normalfont}

\newcommand{\e}{\varepsilon}

\newcommand{\Mod}[1]{\ \mathrm{mod}\ #1}

\newcommand{\cyc}{{\mathrm{cyc}}}

\newcolumntype{?}{!{\vrule width 1pt}}
\newcommand{\lb}{[\![}
\newcommand{\rb}{]\!]}

\newcommand{\cO}{\mathcal{O}}
\newcommand{\cG}{\mathcal{G}}
\newcommand{\cLog}{\mathcal Log}

\newcommand{\RGL}[1]{{\color{blue}#1}}

\newcommand{\FF}{\mathbb{F}}

\newcommand{\ZZ}{\Z}
\newcommand{\Zp}{\Z_p}
\newcommand{\Up}{\Upsilon}
\newcommand{\Tw}{\mathrm{Tw}}

\newtheorem{lthm}{Theorem} 

\title[A canonical construction of signed $p$-adic $L$-functions]{A canonical construction of signed $p$-adic $L$-functions for non-ordinary modular forms of weight $\leq p+1$}

\author[R.~Gajek-Leonard]{Rylan Gajek-Leonard}
\address[Gajek-Leonard]{Department of Mathematics\\
Union College\\
Bailey Hall 206B\\
Schenectady, NY, 12308\\
USA}
\email{gajekler@union.edu}

\author[A.~Lei]{Antonio Lei}
\address[Lei]{
Department of Mathematics and Statistics\\
University of Ottawa\\
150 Louis-Pasteur Pvt\\
Ottawa, ON, K1N 6N5\\
Canada}
\email{antonio.lei@uottawa.ca}


\begin{document}

\begin{abstract}
Fix an odd prime $p$ and let $f$ be a $p$-non-ordinary cuspidal eigen-newform of weight $2\leq k\leq p+1$. We construct a pair of bounded $p$-adic $L$-functions associated to $f$ by decomposing the unbounded $p$-adic $L$-functions in terms of an explicit logarithm-type matrix whose definition does not require $p$-adic Hodge theory. Using this decomposition, we compute asymptotic formulas for the Iwasawa invariants of Mazur--Tate elements attached to non-ordinary forms of weight $\leq p+1$. As a corollary, we obtain a relation between the signed Iwasawa invariants of $p$-non-ordinary and $p$-congruent cuspforms of weights 2 and $p+1$, generalizing previous results in the $a_p=0$ case. 
\end{abstract}

\maketitle

\section{Introduction}

Let $f=\sum_{n\geq1} a_nq^n$ be a cuspidal newform of weight $k\geq 2$, level $\Gamma_1(N)$, and nebentype character $\e$. Fix an odd prime $p\nmid N$ 
and let $K/\Q_p$ denote the extension generated by the images of $a_n$ under a fixed choice of embedding $\iota_p:\overline{\Q}\hookrightarrow\overline{\Q}_p$. We write $\Oo$ for the valuation ring of $K$ and $\FF$ for the residue field. Throughout this article, we assume that $f$ is non-ordinary at $p$, i.e., $a_p\notin\cO^\times$.

Attached to each root of the Hecke polynomial of $f$ at $p$, there is a $p$-adic $L$-function that interpolates the critical $L$-values of $f$ (see \cite{amicevelu75,visik76,MTT}). More precisely, if $\alpha$ and $\beta$ are the two roots of the polynomial $X^2-a_pX+\e(p)p^{k-1}$, there are two $p$-adic distributions $\mu_{f,\alpha}$ and $\mu_{f,\beta}$ on $\Gal(\QQ(\mu_{p^\infty})/\QQ)\cong\Zp^\times$. 
Let $\omega$ be the Teichm\"uller character. Via the Mellin transform, for each $0\le i \le p-2$ and $\Upsilon \in\{\alpha,\beta\}$, the $\omega^i$-isotypic component of $\mu_{f,\Upsilon}$ gives rise to a power series $L_p(f,\Upsilon,\omega^i,X)\in K(\Upsilon)\lb X\rb$.
Unlike the ordinary case, the denominators of these power series can have arbitrarily large $p$-adic valuations. In other words, $\mu_{f,\alpha}$ and $\mu_{f,\beta}$ are not necessarily $p$-adic measures. In \cite{pollack03}, Pollack showed that when $a_p=0$, these power series can be decomposed as linear combinations of elements in $\cO\lb X\rb\otimes K$, which are called plus and minus $p$-adic $L$-functions. The hypothesis $a_p=0$ has been relaxed by Sprung when the weight of $f$ is $2$ in \cite{sprung09,Sprung17}, where a $2\times 2$ logarithmic matrix is constructed to decompose the $p$-adic $L$-functions $L_p(f,\alpha,\omega^i,X)$ and $L_p(f,\beta,\omega^i,X)$ into the sharp and flat $p$-adic $L$-functions $L_p(f,\sharp,\omega^i,X),L_p(f,\flat,\omega^i,X)\in \cO\lb X\rb$.

For higher weight modular forms, similar results were obtained in \cite{LLZ0}. The logarithmic matrix therein is dependent on the choice of a basis of the Wach module of the $p$-adic representation attached to $f$. Although there is no canonical choice of basis, the results in \cite{CL,GL} show that if we make use of an explicit basis given in \cite{berger04}, the Iwasawa invariants of the sharp and flat $p$-adic $L$-functions encode arithmetic information of the Mazur--Tate elements attached to $f$ under the Fontaine--Laffaille condition, that is, $p>k$.

In this article, we construct sharp and flat $p$-adic $L$-functions under the condition that $k \le p+1$ without the use of Wach modules. In particular, our construction can be considered more canonical since it does not depend on the choice of a Wach module basis. 
We assume throughout that $\alpha\neq\beta$.


\begin{lthm}[Theorem~\ref{thm:Lpdecomp}] \label{thmA}
Let $f\in S_k(\Gamma_1(N),\overline\Q_p)$ be a $p$-non-ordinary cuspidal eigen-newform and let $0\leq i\leq p-2$. Assume $ k\leq p+1$.
There exists a $2\times 2$ matrix $\cLog_\infty (f)$ with coefficients in $K(\alpha,\beta)\lb X\rb$
and functions $L_p(f,\sharp,\omega^i,X),L_p(f,\flat,\omega^i,X)\in\Oo\lb X\rb \otimes K$ such that 
\[
\frac{1}{\alpha-\beta}
\begin{bmatrix} L_p(f,\alpha,\omega^i,X)\\L_p(f,\beta,\omega^i,X)\end{bmatrix}
=\cLog_\infty (f) \begin{bmatrix} L_p(f,\sharp,\omega^i,X)\\L_p(f,\flat,\omega^i,X)\end{bmatrix}.
\]
\end{lthm}

Being elements of $\Oo\lb X\rb \otimes K$, the $p$-adic $L$-functions $L_p(f,\sharp,\omega^i,X)$ and $L_p(f,\flat,\omega^i,X)$ have well-defined Iwasawa invariants, which we denote by $\mu(f,\bullet,\omega^i)$ and $\lambda(f,\bullet,\omega^i)$, where $\bullet\in\{\sharp,\flat\}$. These invariants measure the integrality and the number of zeros inside the $p$-adic open unit disk. Similar to \cite[Theorem~A]{GL}, we can use Theorem~\ref{thmA} to relate the sharp and flat $p$-adic $L$-functions to the $p$-adic Mazur-Tate elements associated with $f$, which we denote by $\Theta_n(f,\omega^i)\in K[G_n]$ where $G_n$ is  the Galois group of the $n$-th layer of the cyclotomic $\Z_p$-extension of $\Q$.


\begin{lthm}[Theorem~\ref{thm:FL}]
\label{thmB} Let $f\in S_k(\Gamma_1(N),\overline\Q_p)$ be a $p$-non-ordinary cuspidal eigen-newform. If $2\leq k\leq p+1$ and $\mu(f,\sharp,\omega^i)=\mu(f,\flat,\omega^i)\neq\infty$ then for $n\gg0$ we have
\begin{align*}
\mu\big(\Theta_{n}(f,\omega^i)\big)&=\mu(f,\bullet, \omega^i),\quad \text{and}\\
\lambda\big(\Theta_{n}(f,\omega^i)\big)&= (k-1)q_n+\lambda(f,\bullet,\omega^i),
\end{align*}
where $\bullet=\flat$ if $n$ is odd, $\bullet=\sharp$ if $n$ is even, and 
\[
q_n=\begin{cases}p^{n-1}-p^{n-2}+\cdots +p-1 & \quad\text{if $n\ge2$ is even,}\\
p^{n-1}-p^{n-2}+\cdots +p^2-p & \quad\text{if $n\ge3$ is odd.}
\end{cases}
\]
\end{lthm}

\begin{example}\nf Consider the newform $f\in S_{4}(\Gamma_0(26))$ with LMFDB label \href{https://www.lmfdb.org/ModularForm/GL2/Q/holomorphic/26/4/a/a/}{\texttt{26.4.a.a}} and $q$-series 
\[
f=q - 2 q^2 + 3 q^3 + 4 q^4 + 11 q^5 - 6 q^6 + 19 q^7 +O(q^8).
\]
For $\bullet\in \{\sharp,\flat\}$, let $\mu(f,\bullet)$ and $\lambda(f,\bullet)$ denote the Iwasawa invariants of the trivial isotypic components of the signed 3-adic $L$-functions for $f$.
From a computer calculation in Magma, we find that the values of the sequence
$\lambda(\Theta_n(f,\omega^0))$ for $0\leq n\leq 6$ are 
\[
0,2,6,20,60,182,546.
\]
The $\mu$-invariants are zero for these values of $n$.
Theorem~\ref{thmB} now suggests that $\mu(f,\sharp)=\mu(f,\flat)=0$, $\lambda(f,\sharp)=0$, and $\lambda(f,\flat)=2$. 
\end{example}

\begin{remark}\nf  
\begin{enumerate}
\item In the $p$-ordinary case, the Iwasawa invariants of the Mazur-Tate elements converge to those of the associated $p$-adic $L$-function under the assumption that the residual representation at $p$ is irreducible and the $\mu$-invariant of the $p$-adic $L$-function vanishes. See \cite[Proposition 3.7]{PW}. Recent results in the residually reducible case can be found in \cite{LPP}. 

\item For $p$-non-ordinary modular forms of weight $k>p+1$, the associated Mazur-Tate elements do not in general follow the pattern of Theorem~\ref{thmB}. In particular, at each weight $k>p+1$ there appear to be several  possibilities for the growth of $\lambda(\Theta_n)$. Many examples can be found in \cite{PW,GL,GLaif}. 
\end{enumerate}
\end{remark}

We remark that the case where $k=p+1$ is not covered by the works \cite{CL,GL}. This case is of particular interest since, by \cite{FontaineEdixhoven92}, it is the only weight in the interval $2\leq k\leq p+1$ for which the Serre weight of the residual representation attached to $f$ is something other than the weight of $f$ itself. More precisely, when $k=p+1$ the Serre weight is 2 and it is this fact which ultimately allows us to relate the sharp and flat invariants attached to pairs of modular forms of weights 2 and $p+1$ with isomorphic residual representations.
A similar result was obtained in \cite[Corollary 6.1]{GLaif} under the assumption $a_p(f)=0$. Examples illustrating Theorem~\ref{thmC} can be found in Table~\ref{tablep1}.


\begin{lthm}[Theorem~\ref{thm:p+1}]\label{thmC} Let $f\in S_{p+1}(\Gamma_0(N),\overline{\Q_p})$ and assume $\mu(f,\sharp,\omega^i)=\mu(f,\flat,\omega^i)=0$. There exists a $p$-non-ordinary eigenform $g\in S_2(\Gamma_0(N),\overline{\Q_p})$ with $\overline\rho_f\cong \overline\rho_g$ such that if $\mu(g,\sharp,\omega^i)=\mu(g,\flat,\omega^i)=0$ then 
\begin{align*}
\lambda\big(f,\sharp,\omega^i)&=\lambda\big(g,\flat,\omega^i), \quad \text{and}\\
\lambda\big(f,\flat,\omega^i)&=\lambda\big(g,\sharp,\omega^i)+p-1.
\end{align*}
\end{lthm}

In \cite{EPW}, it is shown that for pairs of $p$-ordinary and $p$-stabilized newforms with irreducible and isomorphic residual representations, the difference $\lambda(f)-\lambda(g)$ depends only on local terms coming from Euler factors at primes dividing their levels, and is in particular independent of weight. The above theorem shows that this weight-independence can fail for $\lambda$-invariants of signed $p$-adic $L$-functions in the $p$-non-ordinary case. 

\subsection*{Acknowledgement}
The authors thank Raiza Corpuz, Jeffrey Hatley, Ariel Pacetti and Rob Pollack for interesting discussions during the preparation of this article. AL's research is supported by the NSERC Discovery Grants Program RGPIN-2026-04351. Parts of this work were carried out during RGL's visit to the University of Ottawa in 2026 supported by the Fields Opportunity for Collaborations US (FOCUS) program. 

\section{An explicit logarithm matrix when $k\leq p+1$}

\subsection{Preliminaries}\label{section:prelim} In this section, we let $K$ denote an arbitrary finite extension of $\Q_p$. For $n\geq 1$, we write $\Phi_n(X)=\sum_{t=0}^{p-1}(1+X)^{tp^{n-1}}$ for the $p^n$-th cyclotomic polynomial evaluated at $1+X$ and set $\Phi_0(X)=X$. Let $\omega_n(X)=(1+X)^{p^n}-1$ and recall the $p$-adic logarithm 
\[
\log_p(1+X)=X\prod_{n\geq 1}\frac{\Phi_n(X)}{p}.
\]
Let $\Tw:K\lb X\rb \rightarrow K\lb X\rb $ be the ring automorphism given by 
\[
\Tw(F(X))=F(u(1+X)-1), \qquad F\in K\lb X\rb ,
\]
where $u$ is a fixed topological generator of $1+p\Zp$,
and set $\Tw^j(F(X))=F(u^j(1+X)-1)$ for $j\in \Z$. For a positive integer $h$, we define
\begin{align*}
\Phi_{n,h}(X) &=\prod_{j=0}^{h-1}\Tw^{-j}(\Phi_n(X)),\\
\omega_{n,h}(X) &=\prod_{j=0}^{h-1}\Tw^{-j}(\omega_n(X)),\\
\log_{p,h}(1+X)&=\prod_{j=0}^{h-1}\Tw^{-j}(\log_p(1+X)).
\end{align*}
Note that $\Phi_{n,h}$ and $\omega_{n,h}$ have coefficients in $\Z_p$. Furthermore, we have the identity 
\[
\log_{p,h}(1+X)=\Phi_{0,h}\prod_{n\geq 1}\frac{\Phi_{n,h}(X)}{p}.
\]
For the remainder of \S\ref{section:prelim}, we fix a sequence $(Q_{n,j})_{n\geq 0}$, $0\leq j\leq h-1$, of polynomials in $K[X]$. 

\begin{definition}\label{def:delta_Pn}Define 
\[
\delta_{n,j}(X) = \sum_{t=0}^j(-1)^{j-t}\binom{j}{t}\Tw^{-t}(Q_{n,t}(X)).
\]
Using the Chinese Remainder Theorem, let $P_{n,h}(X)\in K[X]$  denote the unique polynomial of degree $<hp^n$ satisfying 
\[
P_{n,h}(X)\equiv \Tw^{-j}(Q_{n,j}(X))\Mod \Tw^{-j}(\omega_n(X))
\]
for all $0\leq j\leq h-1$. Furthermore, define $S_n(X,Y)\in K[X,Y]$ and $s_{n,j}(X)\in K[X]$ by
\begin{equation}
S_n(X,Y) = \sum_{j=0}^{h-1} \Tw^{-j}(Q_{n,j})\prod_{\substack{0 \leq t \leq h-1\\ t \neq j}} \frac{u^{-tp^n}(1 + Y) - 1}{u^{(j-t)p^n} - 1}=\sum_{j=0}^{h-1} s_{n,j}(X) Y^j
\label{eq:S(X,Y)}.
\end{equation}
\end{definition}
\begin{remark}\nf
Note that $S_n(X,Y)$ is the unique polynomial of degree $<h$ in $Y$ with coefficients in $K[X]$ satisfying 
\[
S(X, u^{jp^n} - 1) =\Tw^{-j} (Q_{n,j})\quad \text{and} \quad P_{n,h}(X)=S_n(X, (1+X)^{p^n} - 1).
\]
\end{remark}

For $F=\sum_{n\geq 0}c_nX^n\in K\lb X\rb $, we define $\|F\|=\sup_n |c_n|_p$ and say that $F$ is $O(\log_p^h)$ (resp.,  $o(\log_p^h)$) if $\sup_{n}\frac{|c_n|_p}{n^h}<\infty$ (resp., $\lim_{n\rightarrow\infty} \frac{|c_n|_p}{n^h}=0$). 

\begin{lemma}\label{lemma:normbound} Let $F,G\in K[X]$ and suppose $F\equiv G\Mod\omega_{n,h}$ with $\deg G<hp^n$. Then $\|G\|\leq \|F\|$. 
\end{lemma}
\begin{proof}
Choose a coefficient $c$ of $F$ such that $\|F\| =|c|_p$. Then $c^{-1}F\in \cO[X]$. Since the leading coefficient of $\omega_{n,h}$ is a unit, we can use the division algorithm in $\cO[X]$ to find unique $Q,R\in \cO[X]$ such that $\deg R<hp^n$ and
\[
c^{-1}F=Q\omega_{n,h}+R. 
\]
From the uniqueness of the division algorithm in $K[X]$, it follows that $G=cR$ and therefore $\| G\|\leq |c|_p=\|F\|$ as desired. 
\end{proof}

\begin{lemma}\label{lem:sup} Fix a real number $\rho$ satisfying $p^{-1}<\rho < p^{-1/(p-1)}$. Then, 
\begin{enumerate}
\item $\|P_{n,h}\|\leq \sup_{0\leq j\leq h-1}\|s_{n,j}\|$ and
\item if $h\leq p$ then $\sup_{0\leq j\leq h-1} \{\|s_{n,j}\|\rho^j\}= \sup_{0\leq j\leq h-1}\{\|p^{-(n+1)j}\delta_{n,j}\|\rho^j\}$.
\end{enumerate}
\end{lemma}
\begin{proof} Part (1) follows from the fact that $P_{n,h}(X)=S_n(X,\omega_n)=\sum_{j=0}^{h-1}s_{n,j}(X)\omega_n^j$ and $\|\omega_n\|=1$. 

For Part (2), we follow the argument outlined in \cite[p. 10]{CL}. Define $u_n=u^{p^n}-1$. The lower bound on $\rho$ implies that $\rho/|u_n|_p=\rho p^{n+1}>1$ for $n\geq 0$, and the upper bound allows us to define the following homeomorphism between open disks: 
\begin{align*}
\left\{ y:|y|_p < \rho\right\} &\rightarrow \left\{z:|z|_p <\frac{ \rho}{|u_n|_p}\right\},\quad y\mapsto \log(1+y)/\log u^{p^n},
\end{align*}
with inverse given by $z\mapsto u^{zp^n} - 1$. 
We therefore have the following equality of sup-norms:
\begin{equation}\label{avprop4.4}
	\| S(X,Y) \|_\rho = \| H(X, Z) \|_{\rho/|u_n|_p},
\end{equation}
where 
\[
\|R(X,Y)\|_r:=\sup\left\{|R(x,y)|_{p}:|x|_p\le 1,|y|_p\le r\right\}.
\]
Since $h\leq p$ and $\rho/|u_n|_p>1$, for $0\leq j\leq h-1$ we have
\[
\left\| \binom{Z}{j} \right\|_{\frac{\rho}{|u_n|}}
	= \left\| \frac{Z(Z-1) \hdots (Z - j + 1)}{j!}\right\|_{\frac{\rho}{|u_n|}}
	= \left(\frac{\rho}{|u_n|_p}\right)^j.
\]
As $S_n(X,Y) = \sum_{j=0}^{h-1} s_{n,j}(X) Y^j$, we now find that 
\begin{align*}
\sup_{0\leq j\leq h-1} \{\|s_{n,j}(X)\| \rho^j\}&=\left\| S_{n}(X,Y)\right\|_\rho\\
&=\left\|H_n(X,Z)\right\|_{\frac{\rho}{|u_n|}}\\
&=\sup_{j}\bigg\{\left\|\delta_{n,j}(X)\right\|\left\| \binom{Z}{j} \right\|_{\frac{\rho}{|u_n|}}\bigg\}\\
&=\sup_{j}\bigg\{\left\|\delta_{n,j}(X)\right\|\left(\frac{\rho}{|u_n|_p}\right)^j\bigg\}\\
&=\sup_{j}\{\|p^{-(n+1)j}\delta_{n,j}(X)\|\rho^j\}.
\end{align*}
\end{proof}
We have the following lemma, which has its roots in \cite[Lemme 1.2.2]{PR}.

\begin{lemma}\label{construction_lemma} Suppose that $\sup_{n}\|p^{rn}Q_{n,j}\|<\infty$ for some $r\in [0,h)$ and that for each $n$, there exists $d_n\in \R$ such
\[
\|p^{-j(n+1)}\delta_{n,j}\|\leq d_n
\]
for all $0\leq j\leq h-1$.
Then the following assertions hold:
\begin{enumerate}
\item There exists a constant $c_h$ (independent of $n$) such that  $\|P_{n,h}\|\leq c_hd_n$ for all $n\geq 0$.
\item Suppose  $Q_{n+1,j}\equiv Q_{n,j}\Mod \omega_n$  for all $n\geq 0$. Then,
\begin{enumerate}
\item  $P_{n+1,h}\equiv P_{n,h}\Mod \omega_{n,h}$ for all $n\geq 0$, and
\item if there is some $r\in [0,h)$ such that $d_n=O(p^{rn})$, the sequence $P_{n,h}$ converges to an element $P_{\infty,h}\in K\lb X\rb $ that is $O(\log_p^r)$.
\end{enumerate}
\end{enumerate}
\end{lemma}
\begin{proof} See \cite[Lemmas 2.2--2.3 and Remark 2.4]{BFsuper}. 
\end{proof}

\subsection{A special family of polynomials}
Fix an integer $k\geq 2$. In this section, we apply the results of \S\ref{section:prelim} to the sequence 
\[
Q_{n,j}=p^{k-2}\Phi_n, \qquad 0\leq j\leq k-2.
\] 

\begin{definition} Let $p^{k-2}\tilde\delta_{n,j}=\delta_{n,j}$ and $p^{k-2}\tilde\Phi_{n,k-1}=P_{n,k-1}$, where $\delta_{n,j}$ and $P_{n,k-1}$ are the polynomials associated with the sequence $Q_{n,j}=p^{k-2}\Phi_n$ as in Definition \ref{def:delta_Pn}. Specifically, 
\[
\tilde\delta_{n,j}= \sum_{t=0}^j(-1)^{j-t}\binom{j}{t}\Tw^{-t}(\Phi_n)\in\Z_p[X], 
\]
and $\tilde\Phi_{n,k-1}\in \Q_p[X]$ is the unique polynomial of degree $<(k-1)p^n$ satisfying 
\[
\tilde\Phi_{n,k-1}\equiv \Tw^{-j}(\Phi_n)\Mod \Tw^{-j}(\omega_n)
\]
for all $0\leq j\leq k-2$. 
\end{definition}

It will be helpful to calculate $\tilde\Phi_{n,k-1}$ explicitly. Using the Chinese Remainder Theorem, we have 
\[
\tilde\Phi_{n,k-1}=\sum_{j=0}^{k-2}\Tw^{-j}(\Phi_{n})\prod_{\substack{i=0\\ i\neq j}}^{k-2}\frac{\Tw^{-i}(\omega_{n})}{u^{(j-i)p^n}-1}.
\]
Using the fact that $\omega_n=\Phi_n\omega_{n-1}$, this implies
\begin{equation}\label{explicit_PT}
\tilde\Phi_{n,k-1}=\Phi_{n,k-1}\Psi_{n,k-1} 
\end{equation}
where 
\[
\Psi_{n,k-1}:=\sum_{j=0}^{k-2}\prod_{\substack{i=0\\ i\neq j}}^{k-2}\frac{\Tw^{-i}(\omega_{n-1})}{u^{(j-i)p^n}-1}\in \Q_p[X].
\]

\begin{lemma}\label{PT_integral} There exists a constant $e_{k}\in \Z$ such that $p^{e_{k}+k-2}\tilde\Phi_{n,k-1}\in \Z_p[X]$ for all $n\geq 0$. Furthermore, if $k\leq p+1$ then $e_k=0$. 
\end{lemma}
\begin{proof} 
By definition we have
\begin{align*}
\tilde \delta_{n,j}(X)&=p^{k-2}\sum_{t=0}^j(-1)^{j-t}\binom{j}{t}\sum_{s=0}^{p-1}(u^{-t}(1+X))^{sp^{n-1}}\\
&=p^{k-2}\sum_{s=0}^{p-1}(u^{-sp^{n-1}}-1)^j(1+X)^{sp^{n-1}}.
\end{align*}
As $u^{-sp^{n-1}}-1\in p^n\Z_p$, it follows that 
\begin{equation}\label{eqn:delta_norm}
\| \delta_{n,j}(X)\|=p^{-nj-(k-2)}.
\end{equation}
 Therefore,
$
\|p^{-j(n+1)}\delta_{n,j}(X)\|\leq 1.
$
Applying Lemma \ref{construction_lemma} to the sequence $Q_{n,j}=p^{k-2}\Phi_n$, it follows that the there is a constant $c_{k-1}$ independent of $n$ such that 
\[
\|P_{n,k-1}\|=\|p^{k-2}\tilde\Phi_{n,k-1}(X)\|\leq c_{k-1}.
\]

It remains to show that $p^{k-2}\tilde\Phi_{n,k-1}$ has $\Z_p$-coefficients when $k\leq p+1$. Choosing a real number $\rho$ such that $p^{-1}<\rho < p^{-1/(p-1)}$, equation \eqref{eqn:delta_norm} implies 
\[
\sup_{0 \leq j \leq k-2} \{ \| p^{-(n+1)j} \cdot \delta_j(X) \| \cdot \rho^j \}
=\sup_{0 \leq j \leq k-2} \{p^{-(k-2)}(p\rho)^{j}\}=\rho^{k-2}.
\]
As $k\leq p+1$, we use Lemma \ref{lem:sup}(2) to deduce
\[
\|s_{n,j}(X)\|\rho^{j}\le\sup_{0 \leq \ell \leq k-2} \{\| s_{n,\ell}(X) \| \rho^\ell\} \le \rho^{k-2}.
\]
for all $0\leq j\leq k-2$. Hence $\|s_{n,j}(X)\|\leq \rho^{k-2-j}$ and it now follows from Lemma \ref{lem:sup}(1) that 
\[
\|p^{k-2}\tilde\Phi_{n,k-1}\|\leq \sup_{0 \leq j \leq k-2}\rho^{k-2-j}=1.
\]
\end{proof}

\begin{lemma}\label{lem_unit}If $k\leq p+1$ then $p^{k-2}\Psi_{n,k-1}$ is a unit in $\Z_p\lb X\rb $ for all $n\geq 1$. 
 \end{lemma}
  \begin{proof} The $\mu$-invariant of $\Phi_{n,k-1}$ is equal to zero by \cite[Lemma 2.9]{GLaif}, hence $\|\Phi_{n,k-1}\|=1$. It follows from Lemma \ref{PT_integral} and \eqref{explicit_PT} that $\|p^{k-2}\Psi_{n,k-1}\|=\|p^{k-2}\tilde \Phi_{n,k-1}\|\leq 1$. Therefore, it suffices to show that the constant coefficient of $p^{k-2}\Psi_{n,k-1}$ is a $p$-adic unit. 
  
 Let 
 \[
 F_{j}=\displaystyle\prod_{\substack{i=0\\ i\neq j}}^{k-2}\frac{\Tw^{-i}(\omega_{n-1})}{u^{(j-i)p^n}-1}=\displaystyle\prod_{\substack{i=0\\ i\neq j}}^{k-2} \frac{u^{-ip^{n-1}}(1+X)^{p^{n-1}}-1}{u^{(j-i)p^n}-1}.
 \]
Observe that
\[
\Psi_{n,k-1}(0)=F_0(0)
 \]
since $F_j(0)=0$ when $j>0$. Without loss of generality, we assume $u=1+p$ and compute
\[
 \ord_p F_0(0)=\ord_p\bigg(\prod_{i=1}^{k-2}\frac{u^{-ip^{n-1}}-1}{u^{-ip^n}-1}\bigg)=\sum_{i=1}^{k-2}[(n+\ord_p i)-(n+1+\ord_p i)]=-(k-2). 
\]
 \end{proof}

\subsection{Explicit logarithm matrix}
Let $f= \sum a_n(f)q^n\in S_{k}(\Gamma_1(N),\overline\Q_p)$ be as in the introduction and recall that $\alpha$ and $\beta$ denote the roots (which we assume are distinct) of the Hecke polynomial $X^2-a_p(f)X+\e_f(p)p^{k-1}$, where $\e_f$ is the nebentype character associated to $f$. We henceforth let $K/\Q_p$ denote the Hecke field associated to $f$ with respect to our chosen embedding $\overline\Q\hookrightarrow \overline \Q_p$, and define $K'=K(\alpha,\beta).$  

\begin{defn} Define
    \[
    A_f=\begin{bmatrix}
        a_p(f)&1\\ -\epsilon_f(p)p^{k-1}&0
    \end{bmatrix},\quad Q_f=\begin{bmatrix}
        \alpha&-\beta\\-\alpha\beta&\alpha\beta
    \end{bmatrix},\quad 
    A_{n,f}=\begin{bmatrix}
        a_p(f)&1\\-\epsilon_f(p)p^{k-2}\tilde\Phi_{n,k-1}&0
    \end{bmatrix}.
    \]
We also define
\begin{align*}
    C_{n,f}&=A_{n,f}\cdots A_{1,f},\\
    \cLog_n(f)&=\begin{bmatrix}
\frac{1}{\alpha^{n+1}}&0\\ 0&\frac{1}{\beta^{n+1}}
\end{bmatrix}Q_f^{-1}C_{n,f}\in M_{2\times 2}(K'[X]).
\end{align*}
\end{defn}

\begin{remark}\nf Note that the matrices $A_{n,f}$ agree with those defined in \cite[Definition 4.2]{Sprung17} when $k=2$ since $\tilde\Phi_{n,1}=\Phi_n$.
\end{remark}

\begin{lemma}\label{lem:Anf}
    For all $n\ge1$, we have
    \[
    A_{n,f}\equiv A_f\Mod \omega_{n-1,k-1}.
    \]
\end{lemma}
\begin{proof}
    Since $\omega_{n-1,k-1}=\prod_{j=0}^{k-2}\Tw^{-j}(\omega_{n-1})$, it is enough to show that $A_{n,f}\equiv A_f\Mod\Tw^{-j}(\omega_{n-1})$ for each $j\in\{0,\dots,k-2\}$. Note that 
\begin{equation}\label{eq:Phimod}
    \Phi_{n}(u^{-j}(1+X))=\sum_{t=0}^{p-1}\left(u^{-j}(1+X)\right)^{p^{n-1}t}\equiv p\Mod \Tw^{-j}(\omega_{n-1}).
\end{equation}
        By definition,
    \[
    A_{n,f}\equiv\begin{bmatrix}
        a_p(f)&1\\ -\epsilon_f(p)p^{k-2}\Phi_{n}(u^{-j}(1+X))&0
    \end{bmatrix}\equiv \begin{bmatrix}
        a_p(f)&1\\ -\epsilon_f(p)p^{k-1}&0
    \end{bmatrix}\Mod\Tw^{-j}(\omega_{n-1}),
    \]
    from which the lemma follows.
\end{proof}

\begin{proposition}\label{log_matrix} Assuming $f$ is non-ordinary at $p$ and $k\leq p+1$, the matrices $\cLog_n(f)$ converge to a matrix $\cLog_\infty(f)$ defined over $K'\lb X\rb $.  Furthermore, the entries of $\cLog_\infty(f)$ along the first (resp., second) row are $O(\log_p^{\ord_p(\alpha)})$ (resp., $O(\log_p^{\ord_p(\beta)})$). 
\end{proposition}
\begin{proof} We follow the argument of \cite[Lemma 2.8]{BFsuper}. Write 
\[
C_n=
\begin{bmatrix}
 c_{n,1}&c_{n,2}\\ c_{n,3}&{c_{n,4}}
\end{bmatrix}
    \in M_{2\times 2}(\cO[X])\quad\text{and}\quad 
  \cLog_n(f)=
\begin{bmatrix}
\ell_{n,1}&\ell_{n,2}\\ \ell_{n,3}&{\ell_{n,4}}
\end{bmatrix}\in M_{2\times 2}(K'[X]).
\]
Explicitly, we have
 \begin{equation}\label{explicit_Pn}
\begin{bmatrix}
\ell_{n,1}&\ell_{n,2}\\ \ell_{n,3}&{\ell_{n,4}}
\end{bmatrix}
=\frac{1}{\alpha-\beta} 
\begin{bmatrix}
\displaystyle \frac{c_{1,n}}{\alpha^{n+1}}+\frac{c_{3,n}}{\alpha^{n+2}} & \displaystyle \frac{c_{2,n}}{\alpha^{n+1}}+\frac{c_{4,n}}{\alpha^{n+2}}\\ \\ 
\displaystyle \frac{c_{1,n}}{\beta^{n+1}}+\frac{c_{3,n}}{\beta^{n+2}} & \displaystyle \frac{c_{2,n}}{\beta^{n+1}}+\frac{c_{4,n}}{\beta^{n+2}}
\end{bmatrix}.
\end{equation}
Lemma \ref{PT_integral} implies $\|c_{n,i}\|\leq 1$ and since $f$ is non-ordinary, we know that $|\frac{1}{\alpha}|<p^{k-1}$. Therefore, 
\begin{equation}\label{eqn:limits}
\lim_{n\rightarrow\infty}\|p^{n(k-1)}\ell_{n,i}\|=0, \quad i\in \{1,..,4\}. 
\end{equation}
 Furthermore, Lemma \ref{lem:Anf} and the fact that
\[
Q_f^{-1}A_fQ_f=\begin{bmatrix}
 \alpha&0\\ 0&\beta
 \end{bmatrix}
\]
imply 
\begin{equation}
\ell_{n+1,i}\equiv \ell_{n,i}\Mod \omega_{n,k-1}K'[X],\quad \quad i\in \{1,..,4\}. 
\end{equation}
 Finally, it follows from \eqref{explicit_Pn} that there exists a constant $C$ (depending on $f$, but independent of $n$ and $i$) for which 
\[
p^{n\ord_p(\Upsilon)+C}\ell_{n,i}\in \cO[X],
\]
where $\Upsilon=\alpha$ if $i\in \{1,2\}$ and $\Upsilon=\beta$ if $i\in \{3,4\}$.  It now follows from \cite[Lemmas 2.1 and 2.2]{BFsuper} (see also \cite[Proposition IV.4]{amicevelu75} and \cite[\S1.2.1]{perrinriou94}) that for each $i$, the sequence $(\ell_{n,i})_{n\geq1}$ converges to an element $\ell_{\infty,i}\in K'\lb X\rb $ where $\ell_{\infty,1},\ell_{\infty,2}$ are $O(\log_p^{\ord_p(\alpha)})$ while 
$\ell_{\infty,3},\ell_{\infty,4}$ are $O(\log_p^{\ord_p(\beta)})$. 
\end{proof}

The following lemma shows that the matrix $\cLog_\infty(f)$ is closely related to twists of the $p$-adic logarithm. A similar relation holds for Sprung's logarithm matrix at weight two -- see \cite[Remark 4.4]{Sprung17}.

\begin{lemma}\label{lemma:det}$\Phi_{0,k-1}\cdot\det\cLog_\infty(f)$ is a multiple of 
$\log_{p,k-1}(1+X)$. 
\end{lemma}
\begin{proof}Observe that 
\begin{equation}\label{eq:det}
\det\cLog_{n}(f)=\frac{1}{(\alpha\beta)^2(\alpha-\beta)}\cdot\prod_{i=1}^n\Psi_{i,k-1}\prod_{i=1}^n\frac{\Phi_{i,k-1}}{p}.
\end{equation}
The statement now follows by taking the limit as $n\rightarrow\infty$ and using the fact that the determinant is continuous.
\end{proof}

We henceforth write $\Lambda=\Oo\lb X\rb $ and $\Lambda_{n,k-1}=\Lambda/\omega_{n,k-1}\Lambda$.  Define the map 
\begin{align}\label{def:hn}
h_n:\Lambda^2&\rightarrow K\otimes\Lambda^2_{n,k-1}\nonumber,\\
\begin{bmatrix}
        F\\G
    \end{bmatrix}&\mapsto C_{n,f}\begin{bmatrix}
        F\\G
    \end{bmatrix}\Mod \omega_{n,k-1}.
\end{align}
Note that $C_{n+1,f}\equiv A_f C_{n,f}\mod \omega_{n,k-1} $ by Lemma~\ref{lem:Anf}. There is a natural surjection $\Lambda^2/\ker h_{n+1}\rightarrow\Lambda^2/\ker h_n$.

\begin{proposition}\label{prop:lambda^2}
 The natural $\Lambda$-morphism
\[
\Lambda^2\to\varprojlim \Lambda^2/\ker h_n
\]
is surjective. If $k\leq p+1$, then it is an isomorphism.
\end{proposition}
\begin{proof}
 Note that $\omega_{n,k-1}\Lambda^2\subseteq \ker h_n$. This gives the surjection
\[
\varprojlim \Lambda^2/\omega_{n,k-1}\Lambda^2\rightarrow \varprojlim \Lambda^2/\ker h_n.
\]
As $\bigcap(\omega_{n,k-1})=0$, $(\omega_{n,k-1})\subseteq(\omega_n)$ and $\varprojlim\Lambda/\omega_n\Lambda=\Lambda$, we have $\varprojlim \Lambda^2/\omega_{n,k-1}\Lambda^2=\Lambda^2$. 
 This implies the first assertion of the proposition.
 
When $k\leq p+1$, we have $\bigcap \ker h_n=0$. This follows from an argument similar to \cite[Lemma 2.9]{BFsuper}, which we outline briefly. If there is some $\bar{F}\in \Lambda^2$ such that $C_n\bar{F}\equiv 0\Mod \omega_{n,k-1}$ for all $n\geq 1$, then the entries in $\cLog_\infty(f)\bar F$ must be divisible in $K'\lb X\rb $ by $\omega_{n,k-1}$ for all $n\geq 1$. But then 
\[
\cLog_\infty(f)\bar F=\log_{p,k-1}\bar G
\]
for some $\bar G\in (K'\lb X\rb )^2$.  If $\bar G$ is nonzero then then the entries on the right side are $O(\log_p^{k-1})$. However, by Proposition \ref{log_matrix} and the fact that $0<\ord_p\Upsilon<k-1$ for $\Upsilon\in \{\alpha,\beta\}$, the entries on the left side are $o(\log_p^{k-1})$. It follows that $\bar G=0$ and therefore $\cLog_\infty(f)\bar F=0$. Now note that when $k\leq p+1$, we can guarantee that $\det \cLog_\infty(f)$ is nonzero by using~\eqref{eq:det} and Lemma~\ref{lem_unit} to check that valuation of the constant coefficient of $\det \cLog_n(f)$ is nonzero for all $n$. It follows that $\ker \cLog_\infty(f)=0$ and therefore $\bar F=0$ as desired. 
\end{proof}


\section{Construction of signed $p$-adic $L$-functions}

\subsection{Modular symbols and Mazur--Tate elements}
Let $R$ be a commutative ring. Recall that the space of $R$-valued modular symbols of degree $k-2$ (with respect to $\Gamma=\Gamma_1(N)$) is defined by
\[
\Hom_\Gamma(\Delta^0, V_{k-2}(R))=\{\xi:\Delta^0\rightarrow V_{k-2}(R)\mid \xi|\gamma=\xi \,\,\text{for all $\gamma\in \Gamma$}\},
\]
where $\Delta^0=\Div^0(\mathbb{P}^1(\Q))$ and $V_{k-2}(R)$ denotes the space of degree $k-2$ homogeneous polynomials over $K$ in the variables $X$ and $Y$. The right action of $\Gamma$ on a map $\xi: \Delta^0\rightarrow V_{k-2}(R)$ is defined by
\[
(\xi|\gamma)(D)=\xi(\gamma D)|\gamma, \quad D\in \Delta^0,
\]
where $\Gamma$ acts (on the left) on $\Delta^0$ by M\"obius transformations and (on the right) on an element $P(X,Y)\in V_{k-2}(R)$ by
\[
P(X,Y)|\gamma=P(dX-cY,-bX+aY), \qquad \gamma= \begin{psmallmatrix} a &b\\ c&d\end{psmallmatrix}.
\]
By \cite[Proposition 4.2]{ashstevens}, there is a canonical Hecke-equivariant isomorphism 
$$
\Hom_\Gamma(\Delta^0, V_{k-2}(R))\cong H^1_c(\Gamma,V_{k-2}(R)),
$$
and we henceforth identify the two spaces. 

The complex modular symbol $\xi_f\in H^1_c(\Gamma_1(N),V_{k-2}(\C))$ attached to an eigenform $f\in S_k(\Gamma_1(N),\C)$ is defined by
\begin{align*}
\xi_f(\{r\}-\{s\})=2\pi i \int_s^r f(z)(zX+Y)^{k-2}dz.
\end{align*}
As in the introduction, we let $K/\Q_p$ denote a finite extension containing the images of all $a_n$ under our fixed choice of embedding $\iota_p: \overline\Q\hookrightarrow \overline{\Q}_p$. Let $\cO$ be the valuation ring of $K$ choose a uniformizer $\varpi\in K$. 
Writing $\xi_f=\xi_f^++\xi_f^-$, where  $\xi_f^\pm$ lie in the $(\pm1)$-eigenspace for $\begin{psmallmatrix} -1 &0\\ 0&1\end{psmallmatrix}$, it follows from \cite[Proposition 5.11]{PasolPopa} 
 that there exist periods $\Omega^\pm_f\in \C$ such that the modular symbols $\varphi_f^\pm:=\xi_f^\pm/\Omega_f^\pm$ take values in $V_{k-2}(K)$. 
 We henceforth assume that  $\varphi_f^\pm$ 
are \emph{cohomological} (with respect to $\iota_p$) in the sense of \cite[\S2.2]{PW}. By definition, this means that $\varphi_f^\pm$ are scaled by a power of a uniformizer so as to be defined over $V_{k-2}(\cO)$ and there exists a divisor $D^\pm\in \Delta^0$ for which $\varphi_f^\pm(D^\pm)$ has a coefficient in $\cO^\times$. Thus, $\varphi_f^+$ and $\varphi_f^-$ depend on the choice of embedding $\overline \Q\hookrightarrow \overline{\Q}_p$ and are well-defined up to units in $\Oo$. 

Let $\mathcal{G}_n$ $= \mathrm{Gal}(\mathbb{Q}(\mu_{p^n})/\Q)$ and recall the isomorphism 
$\mathcal{G}_n \to $$\left(\ZZ/p^n\ZZ\right)^\times$ given by $\sigma_a \mapsto a$, where $\sigma_a$ is the automorphism $\zeta \mapsto \zeta^a$.

\begin{definition}\label{def:MT} The \textbf{Mazur--Tate elements} $\vartheta_n^\pm(f)$ of level $n \geq 1$ associated to $f$ are defined by 
 \[
 \vartheta_n^\pm(f) = \sum_{a \in (\Z/p^n\Z)^\times} \varphi^\pm_f\,\bigg|\, \begin{pmatrix}
    1 & -a\\ 0 & p^n
\end{pmatrix} (\{\infty\} - \{0\})\cdot \sigma_a \in \cO[X,Y][\mathcal{G}_n].
\]
    We write
    \[
    \vartheta_n^\pm(f)=\sum_{j=0}^{k-2}\binom{k-2}{j}X^jY^{k-2-j}\vartheta_{n,j}(f),
    \]
    where $\binom{k-2}{j}\vartheta_{n,j}^\pm(f)\in \cO[\cG_n]$.
\end{definition}

We have the canonical decomposition
\begin{equation}\label{Gdecomp}
\mathcal{G}_{n+1} \cong \Delta \times G_n,
\end{equation}
where $\Delta \cong (\ZZ/p\ZZ)^\times$ and $G_n$ is a cyclic group of order $p^n$. Letting $\omega:\Delta\rightarrow \Z_p^\times$ be the Teichm\"{u}ller character, we obtain an induced map 
\[
\omega^i: \Oo[\mathcal{G}_{n+1}] \to \Oo[G_n],\quad \sigma_a\mapsto \omega^i(\delta(a))\overline\sigma_a
\]
 for each $0 \leq i \leq p-2$. Here $\sigma_a=\delta(a)\overline\sigma_a$ where $\delta(a)\in \Delta$ and $\overline\sigma_a\in G_n$ are the projections of $\sigma_a$ onto each factor in the decomposition \eqref{Gdecomp}.

 Recall $K'=K(\alpha,\beta)$, where $\alpha$ and $\beta$ are the roots of the Hecke polynomial for $f$ at $p$, and write $\cO'$ for the ring of integers of $K'$.

\begin{defn}\label{def:MT2}
For integers $0\le i\le p-2$, $0\le j\le k-2$, and $n\ge0$, we define
$\Theta_{n,j}(f,\omega^i)\in K[G_n]$ as the image of $\vartheta_{n+1,j}^{\sgn(-1)^i}(f)\in K[\cG_{n+1}]$ under $\omega^{i-j}$. When $n\ge 1$, we define for $\Upsilon\in\{\alpha,\beta\}$ 
\[
\Theta_{n,j}(f,\Up,\omega^i)=
\frac{1}{\Up^{n+1}}\cdot\Theta_{n,j}(f,\omega^i)-\frac{\epsilon(p)p^{k-2}}{\Up^{n+2}}\cdot\nu^n_{n-1}\Theta_{n-1,j}(f,\omega^i)\in K'[G_{n}], 
\]
where $\nu^n_{n-1}:\cO[G_{n-1}]\rightarrow\cO[G_{n}]$ is the norm map that sends $\sigma\in G_{n-1}$ to the sum of the pre-images of $\sigma$ in $G_n$ under the projection map $\pi^n_{n-1} :G_n\rightarrow G_{n-1}$. 
\end{defn}

\begin{remark}\nf  When $j=0$  we simply write $\Theta_n(f,\omega^i)$ to denote $\Theta_{n,0}(f,\omega^i)$. When $i=j=0$, we write $\Theta_n(f)$ for $\Theta_{n,0}(f,\omega^0)$. 
\end{remark}

Let $\gamma$ be a topological generator of $G_\infty=\displaystyle\lim_{\leftarrow} G_n\cong \Z_p$ and recall $\omega_n=(1+X)^{p^n}-1$. There are natural isomorphisms
\begin{equation}\label{eq:identification}
K[G_n] \cong K\lb \gamma-1 \rb / (\gamma^{p^n}-1)\cong K\lb X\rb /(\omega_n),
\end{equation}
the latter being induced by $\gamma \mapsto X+1$. 
Note that $G_n$ is a cyclic group of order $p^n$ generated by the image of $\gamma$ in $G_n$. Therefore, under the above identification the element $\overline\sigma_a\in K[G_n]$ can be regarded as the polynomial $(1+X)^{m(a)}$, where $m(a)$ is the unique integer such that $0\le m(a)\le p^n-1$ and $\overline\sigma_a=\gamma^{m(a)}\mod \gamma^{p^n}$.

Define $Q_{n,j}(f,\omega^i)\in K[X]$  to be the image of the theta element $\Theta_{n,j}(f,\omega^{i})$ under the identification \eqref{eq:identification}.
We similarly define $Q_{n,j}(f,\Upsilon,\omega^i)\in K'[X]$ using $\Theta_{n,j}(f,\Up,\omega^i)$. Furthermore, in what follows we let $u=\chi_\cyc(\gamma)$ where $\chi_\cyc:\cG_\infty\rightarrow \Z_p^\times$ is the $p$-adic cyclotomic character.

\begin{lemma}\label{MT_norm_bound} For all $0\leq j\leq k-2$ we have 
\[
\bigg\|p^{-j(n+1)}\sum_{t=0}^j(-1)^{j-t}\binom{j}{t}\Tw^{-t}(Q_{n,t}(f,\omega^i))\bigg\|\leq 1
\]
\end{lemma}
\begin{proof} This is \cite[Lemma 3.4]{CL}. (Note that this lemma holds without any assumption on the weight of $f$.)
\end{proof}

\subsection{$p$-adic $L$-functions}

Let $P_n(f,\omega^i)\in K[X]$ be the unique polynomial of degree $<(k-1)p^n$ such that 
\[
P_n(f,\omega^i)\equiv \Tw^{-j}(Q_{n,j}(f,\omega^i))\Mod \Tw^{-j}(\omega_n)
\]
for all $0\leq j\leq k-2$.  By Lemma \ref{MT_norm_bound} and Lemma \ref{construction_lemma}, there is a non-negative constant $s$ (depending on $f$ but not on $n$) such that 
\begin{equation}\label{eq:P_nbound}
P_n(f,\omega^i)\in p^{-s}\Oo[X]
\end{equation} 
for all $n\geq 0$. 
Furthermore, the three-term-relation for Mazur--Tate elements (see \cite[Proposition 2.5]{PW}) yields
\begin{equation}\label{eq:Pn3term}
P_{n+1}(f,\omega^i)\equiv a_p(f)P_n(f,\omega^i)-\epsilon_f(p)p^{k-2}\tilde \Phi_{n,k-1}P_{n-1}(f,\omega^i)\Mod \omega_{n,k-1}.
\end{equation}

\begin{theorem}\label{sharp/flat_n} 
Assume that $f\in S_k(\Gamma_1(N),\overline\Q_p)$ is non-ordinary at $p$. 
The polynomials $P_n(f,\sharp,\omega^i)$ and $P_n(f,\flat,\omega^i)$ defined by 
    \[
    \begin{bmatrix}
        P_n(f,\omega^i)\\-\epsilon_f(p)p^{k-2}\tilde\Phi_{n,k-1}P_{n-1}(f,\omega^i)
    \end{bmatrix}= C_{n,f}\begin{bmatrix}
        P_n(f,\sharp,\omega^i)\\P_n(f,\flat,\omega^i)
    \end{bmatrix}.
    \]
  have coefficients in $p^{-e_k-s}\cO[X]$. Furthermore, if $k\leq p+1$ then the sequence $(P_n(f,\sharp,\omega^i),P_n(f,\flat,\omega^i))_{n\ge1}$ converges to an element  $(L_p(f,\sharp,\omega^i),L_p(f,\flat,\omega^i))\in p^{-s}\Lambda^2$.
\end{theorem}

\begin{proof} For simplicity, we omit $\omega^i$ from our notation and write $P_n(f)$ and $P_n(f,\sharp/\flat)$ for the corresponding elements at $\omega^i$. Let $\tilde A_{n,f}=\begin{bmatrix}
        0&-1\\\epsilon_f(p)p^{k-2}\tilde\Phi_{n,k-1}&a_p(f)
    \end{bmatrix}$ be the adjugate matrix of $A_{n,f}$ and write $\tilde C_{n,f}=\tilde A_{1,f}\cdots\tilde A_{n,f}$ so that
    \[
    \tilde C_{n,f} \cdot C_{n,f}=C_{n,f}\cdot\tilde C_{n,f}=\det C_{n,f}.
    \]
 To show that $P_n(f,\sharp),P_n(f,\sharp)\in p^{-e_k-s}\cO[X]$, it suffices to prove that 
    \begin{equation}\label{eq:to-prove}
        \tilde C_{n,f}\begin{bmatrix}
             P_n(f)\\-\epsilon_f(p)p^{k-2}\tilde\Phi_{n,k-1}P_{n-1}(f)
        \end{bmatrix}\in\det(C_{n,f})p^{-e_k-s}\Lambda^2.
    \end{equation}
 We proceed by induction. When $n=1$,
      \[
        \tilde C_{1,f}\begin{bmatrix}
             P_1(f)\\-\epsilon_f(p)p^{k-2}\tilde\Phi_{1,k-1}P_{0}(f)
        \end{bmatrix}=\epsilon_f(p)p^{k-2}\tilde\Phi_{1,k-1}\begin{bmatrix}
            P_0(f)\\P_1-a_p(f)P_0(f)
        \end{bmatrix}.
    \]
   As $P_n(f)\in p^{-s}\cO[X]$ and $p^{k-2}\tilde\Phi_{n,k-1}\in p^{-e_k}\cO[X]$, the inclusion \eqref{eq:to-prove} holds. Suppose that \eqref{eq:to-prove} holds for some $n\ge1$. Then
    \begin{align*}
    \tilde C_{n+1,f}\begin{bmatrix}
             P_{n+1}(f)\\-\epsilon_f(p)p^{k-2}\tilde\Phi_{n+1,k-1}P_{n}(f)
        \end{bmatrix}&=\tilde C_{n,f}\tilde A_{n+1,f}\begin{bmatrix}
             P_{n+1}(f)\\-\epsilon_f(p)p^{k-2}\tilde\Phi_{n+1,k-1}P_{n}(f)
        \end{bmatrix}\\
        &=\epsilon_f(p)p^{k-2}\tilde\Phi_{n+1,k-1}\tilde C_{n,f}\begin{bmatrix}
            P_{n}(f)\\P_{n+1}-a_p(f)P_n(f)
        \end{bmatrix}\\
        &=\det A_{n+1,f}\cdot\tilde C_{n,f}\begin{bmatrix}
            P_{n}(f)\\-\epsilon_f(p)p^{k-2}\tilde\Phi_{n,k-1}P_{n-1}(f)
        \end{bmatrix}.
    \end{align*}
    Thus, applying the inductive hypothesis implies that \eqref{eq:to-prove} holds for $n+1$.

Now assume $k\leq p+1$. Then we may take $e_k=0$ by Lemma \ref{PT_integral}. By Proposition~\ref{prop:lambda^2}, it remains to show that $( P_n(f,\sharp),P_n(f,\flat))_{n\geq 1}$ defines an element of $p^{-s}\varprojlim\Lambda^2/\ker h_n.$ Observe that 
\begin{align}
C_{n,f}\begin{bmatrix}
P_{n+1}(f,\sharp)\\P_{n+1}(f,\flat)
 \end{bmatrix}&=A_{n+1,f}^{-1}C_{n+1,f}\begin{bmatrix}
P_{n+1}(f,\sharp)\\P_{n+1}(f,\flat)
 \end{bmatrix}\nonumber \\
 &=A_{n+1,f}^{-1}\begin{bmatrix}P_{n+1}(f)\\-\epsilon_f(p)p^{k-2}\tilde\Phi_{n+1,k-1}P_{n}(f)\end{bmatrix}\nonumber\\
 &\equiv A_f^{-1}\begin{bmatrix}P_{n+1}(f)\\-\epsilon_f(p)p^{k-2}\tilde\Phi_{n+1,k-1}P_{n}(f)\end{bmatrix}\Mod \omega_{n,k-1}\label{eq:step2}\\
  &\equiv  \begin{bmatrix}P_{n}(f)\\ P_{n+1}(f)-a_p(f)P_n(f)\end{bmatrix}\Mod \omega_{n,k-1}\label{eq:step3}\\
  &\equiv  \begin{bmatrix}P_{n}(f)\\ -\e(p)p^{k-2}\tilde\Phi_{n,k-1}P_{n-1}(f)\end{bmatrix}\Mod \omega_{n,k-1}\label{eq:step4}\\
  &\equiv C_{n,f}\begin{bmatrix}
P_{n}(f,\sharp)\\P_{n}(f,\flat)
 \end{bmatrix}\Mod \omega_{n,k-1},\nonumber
\end{align}
where \eqref{eq:step2}, \eqref{eq:step3}, and \eqref{eq:step4}  follow from Lemma~\ref{lem:Anf}, \eqref{eq:Phimod}, and \eqref{eq:Pn3term}, respectively.
    \end{proof}

\begin{corollary}\label{cor:thetaPn} For all $n\ge0$ and $j\in\{0,\dots, k-2\}$, we have  
 
\[
\begin{bmatrix}
Q_{n,j}(f,\omega^i)\\
-\epsilon_f(p)p^{k-2}\Phi_{n}Q_{n-1,j}(f,\omega^i)
\end{bmatrix}
\equiv 
\Tw^{j}(C_{n,f})
\begin{bmatrix}
\Tw^{j}\big(P_n(f,\sharp,\omega^i)\big)\\
\Tw^{j}\big(P_n(f,\flat,\omega^i)\big)
\end{bmatrix}\mod\omega_n
\]
\end{corollary}
\begin{proof} The previous theorem and the definitions of $P_n(f,\omega^i)$ and $\tilde\Phi_{n,k-1}$ imply 
\[
    \begin{bmatrix}
        \Tw^{-j}(Q_{n,j}(f,\omega^i))\\-\epsilon_f(p)p^{k-2} \Tw^{-j}(\Phi_{n}Q_{n-1,j}(f,\omega^i))
    \end{bmatrix}\equiv C_{n,f}\begin{bmatrix}
      P_n(f,\sharp,\omega^i)\\ P_n(f,\flat,\omega^i)
    \end{bmatrix}\mod \Tw^{-j}(\omega_n).
    \]
  Thus, the assertion of the corollary follows. 
\end{proof}

For each $\Up\in\{\alpha,\beta\}$, define $P_{n}(f,\Up,\omega^i)$ as the unique polynomial of degree $<(k-1)p^n$ such that
\[
P_{n}(f,\Up,\omega^i)\equiv Q_{n,j}(f,\Up,\omega^i)(u^{-j}(1+X)-1)\mod\Tw^{-j}(\omega_n)K'[X]
\]
for $0\le j\le k-2$. The calculations in \cite[Lemma 3.8]{CL} show that 
\[
P_n(f,\Upsilon,\omega^i)\equiv \frac{1}{\Upsilon^{n+1}}P_n(f,\omega^i)-\frac{\epsilon(p)p^{k-2}}{\Upsilon^{n+2}}\tilde\Phi_{n,k-1}\cdot P_{n-1}(f,\omega^i)\Mod\omega_{n,k-1}K'[X].
\]
In \cite{CL}, the hypothesis that $K/\Qp$ is unramified is used to show that $P_n(f,\omega^i)$ is defined over $\cO$. Without this hypothesis, the polynomials $P_{n}(f,\omega^i)$ might not be integral, but from ~\eqref{eq:P_nbound} we can see that they do have uniformly bounded denominators. 
 It then follows from the above congruence together with Lemmas~\ref{lemma:normbound} and \ref{PT_integral} that
 \[
\|P_n(f,\Upsilon,\omega^i)\|\leq C \max\left\{ \left\|\frac{1}{\Upsilon^{n+1}}\right\|,\left\|\frac{1}{\Upsilon^{n+2}}\right\|\right\}
\]
for some constant $C$ independent of $n$. Therefore,  
 $\| p^{\ord_p(\Up)n}P_{n}(f,\Up,\omega^i) \| = O(1)$, and we can apply Lemma~\ref{construction_lemma}  after setting $h=k-1$, $r=\ord_p(\Up)$, and $Q_{n,j}=Q_{n,j}(f,\Up,\omega^i)$. In particular, we can define $L_{p}(f,\Up,\omega^i,X)\in K'\lb X\rb$ as the limit of the polynomials $P_n(f,\Up,\omega^i)$.
In other words,
\begin{equation}\label{eq:limit-Lp} 
L_{p}(f,\Up,\omega^i,X) = \lim_{n \to \infty} P_n(f,\Up,\omega^i),
\end{equation}
which is $O(\log^{\ord_p(\Up)})$. Furthermore,
\begin{equation}
\label{eq:padicL-cong}
L_{p}(f,\Up,\omega^i,X)\equiv P_n(f,\Up,\omega^i)\,\bmod\,\omega_{n,k-1}K'\lb X\rb.    
\end{equation}

\begin{theorem}\label{thm:Lpdecomp} Assume that $f\in S_k(\Gamma_1(N),\overline\Q_p)$ is non-ordinary at $p$ and $ k\leq p+1$.
There exist $L_p(f,\sharp,\omega^i,X),L_p(f,\flat,\omega^i,X)\in\Oo\lb X\rb \otimes K$ such that 
\[
\frac{1}{\alpha-\beta}\begin{bmatrix} L_p(f,\alpha,\omega^i,X)\\L_p(f,\beta,\omega^i,X)\end{bmatrix}
=\cLog_\infty (f)\begin{bmatrix} L_p(f,\sharp,\omega^i,X)\\L_p(f,\flat,\omega^i,X)\end{bmatrix}.
\]
\end{theorem}
\begin{proof} By definition, 
\[
L_p(f,\Upsilon,\omega^i)\equiv P_n(f,\Upsilon,\omega^i)\equiv \frac{1}{\Upsilon^{n+1}}P_n(f,\omega^i)-\frac{\e(p)p^{k-2}}{\Upsilon^{n+2}}\tilde\Phi_{n,k-1}P_{n-1}(f,\omega^i)\Mod \omega_{n,k-1}. 
\]
Therefore, by Theorem \ref{sharp/flat_n} we have
\begin{align*}
\begin{bmatrix}
\alpha^{n+1}L_p(f,\alpha,\omega^i)\\
\beta^{n+1}L_p(f,\beta,\omega^i)
\end{bmatrix}
&\equiv 
\begin{bmatrix}
1&\alpha^{-1}\\
1&\beta^{-1}\\
\end{bmatrix}
\begin{bmatrix}
P_n(f,\omega^i)\\
-\e(p)p^{k-2}\tilde\Phi_{n,k-1}P_{n-1}(f,\omega^i)
\end{bmatrix}
\Mod \omega_{n,k-1}\\
&\equiv
(\alpha-\beta)Q^{-1}
\begin{bmatrix}
P_n(f,\omega^i)\\
-\e(p)p^{k-2}\tilde\Phi_{n,k-1}P_{n-1}(f,\omega^i)
\end{bmatrix}
\Mod \omega_{n,k-1}\\
&\equiv 
(\alpha-\beta)Q^{-1}C_n
\begin{bmatrix}
 P_n(f,\sharp,\omega^i)\\
 P_n(f,\flat,\omega^i)
\end{bmatrix}
\Mod \omega_{n,k-1}.
\end{align*}
It follows that 
\[
\frac{1}{\alpha-\beta}
\begin{bmatrix} L_p(f,\alpha,\omega^i,X)\\
L_p(f,\beta,\omega^i,X)
\end{bmatrix}
\equiv \cLog_n (f)
\begin{bmatrix}
 P_n(f,\sharp,\omega^i)\\
 P_n(f,\flat,\omega^i)
\end{bmatrix}
\Mod \omega_{n,k-1}. 
\]
Letting $n\rightarrow\infty$, the statement now follows from Proposition~\ref{log_matrix} and Theorem~\ref{sharp/flat_n}
\end{proof}

\section{Iwasawa invariants of Mazur--Tate elements}

 Define for integers $h\geq1$ the polynomials 
\begin{align*}
\Phi_{n,h}^+ &= \prod_{2\le m\leq n, \text{even}}\Phi_{m,h},\\
\Phi_{n,h}^- &= \prod_{1\le m\leq n, \text{odd}}\Phi_{m,h},\\
\tilde\Phi_{n,h}^+ &=\prod_{2\leq m\leq n, \text{even}} p^{h-1}\tilde\Phi_{m,h}\\
 \tilde\Phi_{n,h}^- &=\prod_{1\leq m\leq n, \text{odd}} p^{h-1}\tilde\Phi_{m,h}.
\end{align*}

\begin{remark}\label{remark:PT}\nf  When $k\leq p+1$, equation \eqref{explicit_PT} and Lemma \ref{lem_unit} allow us to write 
\[
\tilde\Phi_{n,k-1}^+=*\Phi_{n,k-1}^+\qquad \tilde\Phi_{n,k-1}^-=*\Phi_{n,k-1}^-,
\]
where $*$ denotes a unit in $\Z_p\lb X\rb$.

\end{remark}

\begin{lemma}\label{lem:Cnf} For all $n\geq1$, we have 
\[
C_{n,f}\equiv 
\begin{cases}
 \begin{bmatrix}
0 &(-\e_f(p))^{\frac{n-1}{2}}\tilde\Phi_{n,k-1}^+\\
(-\e_f(p))^{\frac{n-1}{2}}\tilde\Phi_{n,k-1}^- &0
\end{bmatrix}\Mod \varpi \cO\lb X\rb&\quad\text{if $n$ is odd,}\\
\\
 \begin{bmatrix}
(-\e_f(p))^{\frac{n}{2}}\tilde\Phi_{n,k-1}^- &0\\
0&(-\e_f(p))^{\frac{n}{2}}\tilde\Phi_{n,k-1}^+
\end{bmatrix}\Mod \varpi\cO\lb X\rb &\quad\text{if $n$ is even},\\
\end{cases}
\]
\end{lemma}
\begin{proof} Since $a_p(f)\equiv 0\Mod \varpi$, this follows directly from the definition of $C_{n,f}$. 
\end{proof}


\begin{definition} The \emph{Iwasawa invariants} of a nonzero element $F=\sum_{i=0}^\infty c_i X^i\in \cO\lb X\rb\otimes K$ are defined by 
\[
\mu(F)=\min\{\ord_\varpi c_i\mid i\geq 0 \} \quad \text{and}\quad 
\lambda(F)=\min\{i \mid \ord_\varpi c_i=\mu(F)\}.
\]
Similarly, if $\Theta\in K[G_n]$ is nonzero then we can write $\Theta=\sum_{i=0}^{p^n-1} c_i \gamma_n^i$, where $\gamma_n$ generates $G_n$, and define
\[
\mu(\Theta)=\min\{\ord_\varpi c_i\mid i\geq 0 \}\quad \text{and} \quad \lambda(\Theta)=\min\{i \mid \ord_\varpi c_i=\mu(\Theta)\}.
\]
We follow the convention that if $F$ or $\Theta$ is zero then the $\lambda$ and $\mu$-invariants are defined to be $\infty$.
\end{definition}

In what follows, we let $\mu(f,\sharp/\flat,\omega^i)$ and $\lambda(f,\sharp/\flat,\omega^i)$ denote the Iwasawa invariants of the signed $p$-adic $L$-functions $L_p(f,\sharp/\flat,\omega^i,X)$ from Theorem~\ref{sharp/flat_n} and define 
\[
q_n=\begin{cases}p^{n-1}-p^{n-2}+\cdots +p-1 & \quad\text{if $n\ge2$ is even,}\\
p^{n-1}-p^{n-2}+\cdots +p^2-p & \quad\text{if $n\ge3$ is odd.}
\end{cases}
\]

\begin{theorem}\label{thm:FL} Assume that $f\in S_k(\Gamma_1(N),\overline\Q_p)$ is non-ordinary at $p$ and $k\leq p+1$. 
If $\mu(f,\sharp,\omega^i)=\mu(f,\flat,\omega^i)\neq \infty$ then for 
$n\gg0$ we have
\begin{align*}
\mu\big(\Theta_{n}(f,\omega^i)\big)&=\mu(f,\bullet,\omega^i),\quad \text{and}\\
\lambda\big(\Theta_{n}(f,\omega^i)\big)&= (k-1)q_n+\lambda(f,\bullet,\omega^i),
\end{align*}
where $\bullet=\flat$ if $n$ is odd and $\bullet=\sharp$ if $n$ is even. 
\end{theorem}
\begin{proof} Let $a$ denote the common $\mu$-invariant $\mu(f,\sharp,\omega^i)=\mu(f,\flat,\omega^i)$. By definition, we may replace $\Theta_{n,j}(f,\omega^{i})$ by $Q_{n,j}(f,\omega^{i})$. For $n\gg0$, Corollary~\ref{cor:thetaPn} and Theorem~\ref{thm:Lpdecomp} imply
\[
 \begin{bmatrix}
\varpi^{-a}Q_{n,j}(f,\omega^i)\\
-\epsilon_f(p)p^{k-2}\Phi_{n}\varpi^{-a}Q_{n-1,j}(f,\omega^i)
\end{bmatrix}
\equiv
 \Tw^{j}(C_{n,f})
\begin{bmatrix}
\Tw^{j}\big(\varpi^{-a}L_p(f,\sharp,\omega^i)\big)\\
\Tw^{j}\big(\varpi^{-a}L_p(f,\flat,\omega^i)\big)
\end{bmatrix}
\mod\omega_n\Oo\lb X\rb
\]
Reducing modulo the idea $(\varpi,\omega_n)$ of $\cO\lb X\rb$, it follows from Lemma~\ref{lem:Cnf} and Remark \ref{remark:PT} that
\[
\varpi^{-a}Q_{n,j}(f,\omega^i)\equiv
 \begin{cases}
\Tw^j\big(*\Phi_{n,k-1}^+\varpi^{-a}L_p(f,\flat,\omega^i)\big)\Mod (\varpi,\omega_n)&\quad\text{if $n$ is odd,}\\
\Tw^j\big(*\Phi_{n,k-1}^-\varpi^{-a}L_p(f,\sharp,\omega^i)\big)\Mod (\varpi,\omega_n)&\quad\text{if $n$ is even,}
  \end{cases}
\]
Using \cite[Lemmas 2.9 and 3.7]{GLaif} and the fact that multiplying by an element of $\cO\lb X\rb^\times$ does not change Iwasawa invariants, we have 
\begin{align*}
\lambda\big(\Tw^j(*\Phi_{n,k-1}^\circ \varpi^{-a}L_p(f,\bullet,\omega^i))\big)&=(k-1)q_n+\lambda(f,\bullet,\omega^i)\\
\mu\big(\Tw^j(*\Phi_{n,k-1}^\circ \varpi^{-a}L_p(f,\bullet,\omega^i))\big)&=0
\end{align*}
for $(\circ,\bullet)\in \{(+,\flat),(-,\sharp)\}$ and $n$ of parity $-\circ$. Since $k\leq p+1$, the sequence $p^n-(k-1)q_n\rightarrow \infty$ as $n\rightarrow\infty$. Therefore, we can take $n\gg0$ such that $(k-1)q_n+\lambda(f,\bullet,\omega^i)<p^n$.
From \cite[Lemma 2.8]{GL}, it follows that the Iwasawa invariants of $\Tw^j\left(*\Phi_{n,k-1}^\circ \varpi^{-a}L_p(f,\bullet,\omega^i)\right)$ agree with those of $\varpi^{-a}Q_{n,j}(f,\omega^i)$, which concludes the proof.
\end{proof}
\begin{remark}\nf Theorem~\ref{thm:FL} is identical to 
 \cite[Theorem 3.11]{GL}, except that the result proved here also holds at weight $k=p+1$. In \cite{GL}, an essential ingredient is the decomposition \cite[Proposition 2.11]{BFsuper} of the unbounded $p$-adic $L$-functions $L_p(f,\alpha,\omega^i,X)$ and $L_p(f,\alpha,\omega^i,X)$ in terms of a logarithm matrix $M_{log}$ (see \cite[\S2.2]{BFsuper}) coming from a choice of basis of the Wach module associated to $\overline\rho_f|_{G_{\Q_p}}$. It is this choice of basis that requires $k\leq p$. In contrast, the above proof relies only on the decomposition given in Theorem \ref{thm:Lpdecomp}, which does not require  $p$-adic Hodge theory and holds for the additional weight $k=p+1$. 
 

\end{remark}

\subsection{Modular forms of weight $p+1$} 
In this section, we restrict to level $\Gamma_0(N)$ in order to apply results of Pollack and Weston \cite{PW}. The following theorem generalizes \cite[Corollary 6.1]{GLaif}, where an analogous result is proved under the assumption that $a_p(f)=0$. The proof given here is similar to that of \cite{GLaif}. The key idea is that the Iwasawa invariants of a weight $p+1$ modular form $f$ can be calculated in two different ways: one in terms of the $p$-adic $L$-function of $f$ itself and the other in terms of the $p$-adic $L$-function of a lower weight modular form. 

\begin{theorem}\label{thm:p+1} Let $f\in S_{p+1}(\Gamma_0(N),\overline{\Q_p})$ be a $p$-non-ordinary newform and assume that $\mu(f,\sharp,\omega^i)=\mu(f,\flat,\omega^i)\neq\infty$. There exists a $p$-non-ordinary eigenform $g\in S_2(\Gamma_0(N),\overline{\Q_p})$ with $\overline\rho_f\cong \overline\rho_g$ such that if $\mu(g,\sharp,\omega^i)=\mu(g,\flat,\omega^i)\neq\infty$ then the following hold:
\begin{enumerate}
\item $\mu(f,\sharp,\omega^i)=\mu(f,\flat,\omega^i)=0$ if and only if $\mu(g,\sharp,\omega^i)=\mu(g,\flat,\omega^i)=0$.
\item If either of the conditions in (1) are true then 
\begin{align*}
\lambda\big(f,\sharp,\omega^i)&=\lambda\big(g,\flat,\omega^i), \quad \text{and}\\
\lambda\big(f,\flat,\omega^i)&=\lambda\big(g,\sharp,\omega^i)+p-1.
\end{align*}
\end{enumerate}
\end{theorem}
\begin{proof}  
By \cite[Theorem 2.6]{FontaineEdixhoven92}, since $f$ has weight $p+1$ we know that the representation $\overline{\rho_f}|_{G_{\Q_p}}$ is irreducible and has Serre weight two (in the sense of \cite[Definition VII.1.7]{cornell2013modular}). 
It follows from \cite[Corollary 5.3(1)]{PW} that there exists an eigenform $g\in S_2(\Gamma_0(N))$ with $\overline{\rho_f}\cong \overline{\rho_g}$ such that (after applying Theorem \ref{thm:FL} to the $\mu$-invariants of $g$) we have
\begin{equation*}
\text{$\mu(\Theta_n(f,\omega^i))=0$ for $n\gg0$ if and only if $\mu(g,\sharp,\omega^i)=\mu(g,\sharp,\omega^i)=0$.}
\end{equation*}
Assertion (1) now follows from Theorem \ref{thm:FL} (applied to the $\mu$-invariants of $f$). 

For the second assertion, note that if either of the conditions of (1) are true then we can use \cite[Corollary 5.3(2)]{PW} and Theorem \ref{thm:FL} (applied to the $\lambda$-invariants of $g$) to write the $\lambda$-invariants of $\Theta_n(f,\omega^i)$ in terms of those coming from the $p$-adic $L$-function of $g$ as follows:
\begin{equation*}
\lambda(\Theta_n(f,\omega^i))=pq_n+ \begin{cases} \lambda(g,\flat,\omega^i) \quad& \text{if $n\gg0$ is even,}\\
p-1+\lambda(g,\sharp,\omega^i) \quad& \text{if $n\gg0$ is odd.}
\end{cases}
\end{equation*}
On the other hand, applying Theorem \ref{thm:FL} directly to $f$ yields the description
\[
\lambda(\Theta_n(f,\omega^i))=pq_n+\begin{cases}\lambda(f,\sharp,\omega^i)&\quad \text{if $n\gg0$ is even}\\
\lambda(f,\flat,\omega^i)&\quad \text{if $n\gg0$ is odd.}
\end{cases}
\] 
Combining these two expressions completes the proof.
\end{proof}

Numerical examples of Theorem~\ref{thm:p+1} can be found in Table~\ref{tablep1}. Both the Magma and LMFDB labels of each modular form are provided; to initiate the form in Magma, use the command \texttt{Newform("label")}. These examples were calculated in Magma \cite[version 2.27-3]{Magma} and our code is available upon request. 

\begin{table}[h]
\begin{center}
\scalebox{1}{
\begin{tabular}{|c||c|c||c|c|c|c|c|c||c|c||c|c|}
\hline
$p$& $k$ 	& $N$ & 0 & 1 &2&3&4&5& $\lambda^\sharp $& $\lambda^\flat $&Magma & LMFDB\\ 
\hline
\hline
3& 4& 26&0& 2&6 & 20&60 &182  & 0&2 &\texttt{G0N26k4C}&\href{https://www.lmfdb.org/ModularForm/GL2/Q/holomorphic/26/4/a/a/}{\texttt{26.4.a.a}}\\
& 2& 26&0& 0&2 & 6&20 &60 &0 &0 &\texttt{G0N26k2B}&\href{https://www.lmfdb.org/ModularForm/GL2/Q/holomorphic/26/2/a/b/}{\texttt{26.2.a.b}}\\
\hline
3& 4&158  & 0& 0&7 & 21&61 &183&1 &3 &\texttt{G0N158k4A}&\href{https://www.lmfdb.org/ModularForm/GL2/Q/holomorphic/158/4/a/a/}{\texttt{158.4.a.a}}\\
& 2&158  &0 &1 &3 &7 &21 &61 & 1&1 &\texttt{G0N158k2E}&\href{https://www.lmfdb.org/ModularForm/GL2/Q/holomorphic/158/2/a/c/}{\texttt{158.2.a.c}}\\
\hline
3& 4& 392&0&1 & 5& 21& 63&183 &3 &3 & \texttt{G0N392k4E}&\href{https://www.lmfdb.org/ModularForm/GL2/Q/holomorphic/392/4/a/a/}{\texttt{392.4.a.a}}\\
& 2&392  &0&1 &3 & 9& 21& 63& 1& 3& \texttt{G0N392k2E}&\href{https://www.lmfdb.org/ModularForm/GL2/Q/holomorphic/392/2/a/d/}{\texttt{392.2.a.d}}\\
\hline
5& 6&14  &0&4 &20 &104 & 520&2604 &0& 4 &\texttt{G0N14k6A}&\href{https://www.lmfdb.org/ModularForm/GL2/Q/holomorphic/14/6/a/b/}{\texttt{14.6.a.b}}\\
& 2& 14  &0& 0&4 &20 &104 & 520& 0&  0&\texttt{G0N14k2A}&\href{https://www.lmfdb.org/ModularForm/GL2/Q/holomorphic/14/2/a/a/}{\texttt{14.2.a.a}}\\ 
\hline
7&8 &15  & 0& 6& 42& 300& 2100&14706 &0 & 6&\texttt{G0N15k8B}&\href{https://www.lmfdb.org/ModularForm/GL2/Q/holomorphic/15/8/a/a/}{\texttt{15.8.a.a}}\\
& 2&15  &0& 0&6 & 42& 300&2100 & 0& 0& \texttt{G0N15k2A}&\href{https://www.lmfdb.org/ModularForm/GL2/Q/holomorphic/15/2/a/a/}{\texttt{15.2.a.a}}\\
\hline
\end{tabular}
}
\caption{Invariants $\lambda(\Theta_n)$, $0\leq n\leq 5$, attached to rational $p$-non-ordinary newforms of level $\Gamma_0(N)$. Each block contains a pair of forms of weight 2 and $p+1$ such that their residual representations are isomorphic. The signed $\lambda$-invariants are estimated from the Mazur-Tate elements using Theorem \ref{thm:FL}. Note that each pair of signed invariants satisfies the behavior in Theorem \ref{thm:p+1}. }
\label{tablep1}
\end{center}
\end{table}

\bibliographystyle{alpha}
\bibliography{references}
\end{document}